\documentclass[a4paper,11pt]{article}

\usepackage{amsfonts,qtree,stmaryrd,textcomp}
\usepackage{pifont,amssymb,amsmath,amsthm,tabularx,graphicx}
\usepackage{tree-dvips,pictexwd,dcpic}
\usepackage{bbm}
\usepackage{bm}

\usepackage{stmaryrd}

\usepackage[T1]{fontenc}
\usepackage[latin1]{inputenc}
\usepackage[text={17cm,26cm},centering]{geometry}

\usepackage{etex}
\usepackage{fancyhdr}
\usepackage{graphicx}
\usepackage{enumitem}
\usepackage{amsmath}
\usepackage[bbgreekl]{mathbbol}
\usepackage{amssymb}
\usepackage{amsthm}
\usepackage[all]{xy}

\usepackage{epigraph}
\RequirePackage{bussproofs}

\usepackage{mathabx}

\usepackage{scalerel,stackengine}

\usepackage{framed}
\usepackage{mathtools}
\usepackage{url}
\usepackage{proof}
\usepackage{xspace}
\usepackage{listings}
\usepackage{wrapfig}
\usepackage{tikz}
\usepackage{tikz-cd}

\usepackage{relsize}

\usetikzlibrary{positioning}
\usetikzlibrary{shapes}

\usetikzlibrary{arrows}
\usetikzlibrary{calc}
\usepackage{tikz-3dplot}
\usetikzlibrary{decorations.pathmorphing}

\usepackage{amsfonts,amssymb,amsmath,qtree,stmaryrd,textcomp, amsthm}

\usetikzlibrary{decorations.pathreplacing}
\tikzset{axis/.style={&lt;-&gt;}}

\stackMath
\newcommand\reallywidehat[1]{%
\savestack{\tmpbox}{\stretchto{%
  \scaleto{%
    \scalerel*[\widthof{\ensuremath{#1}}]{\kern-.6pt\bigwedge\kern-.6pt}%
    {\rule[-\textheight/2]{1ex}{\textheight}}
  }{\textheight}%
}{0.5ex}}%
\stackon[1pt]{#1}{\tmpbox}%
}
 \definecolor{MyBlue}{rgb}{0.05, 0.25, 0.65}
 
 \definecolor{MyRed}{rgb}{0.90, 0.05, 0.05}
 
\definecolor{MyGreen}{rgb}{0.05, 0.90, 0.05}

\newcommand{\C}[1]{\mathcal{#1}}

\usepackage{ mathrsfs }

\newtheorem{theorem}{Theorem}[section]
\newtheorem{proposition}[theorem]{Proposition}

\newtheorem{definition}[theorem]{Definition}

\newcommand{\CA}{\mathrm{C_{u}(A)}}

\newcommand{\id}{\mathrm{id}}

\newcommand{\AC}{\mathrm{AC}}

\newcommand{\Mor}{\mathrm{Mor}}

\newcommand{\Cat}{\mathrm{\mathbf{Cat}}}

\newcommand{\To}{\Rightarrow}

\newcommand{\MLTT}{\mathrm{MLTT}}

\newcommand{\pr}{\textnormal{\texttt{pr}}}

\newcommand{\BST}{\mathrm{BST}}

\newcommand{\Set}{\mathrm{\mathbf{Set}}}

\newcommand{\op}{\mathrm{op}}

\newcommand{\Fun}{\mathrm{Fun}}

\newcommand{\Ob}{\mathrm{Ob}}

\newcommand{\PSh}{\mathrm{PSh}}

\begin{document}

\date{}

\title{\textbf{From the Sigma-type to the Grothendieck construction}}

\author{Iosif Petrakis\\	
Mathematics Institute, Ludwig-Maximilians-Universit\"{a}t M\"{u}nchen\\
petrakis@math.lmu.de}  

%





\maketitle

\begin{abstract}
\noindent 
We translate properties of the Sigma-type in Martin-L\"of Type Theory $(\MLTT)$ to properties of
the Grothendieck 
construction in category theory. Namely, equivalences in $\MLTT$ that involve the Sigma-type motivate 
isomorphisms
between corresponding categories that involve the Grothendieck construction. The type-theoretic axiom of
choice and 
the ``associativity'' of the Sigma-type are the main examples of this phenomenon that are treated here.

\textit{Keywords}: Martin-L\"of type theory, type-theoretic axiom of choice, category theory, 
Grothendieck construction
\end{abstract}


\section{Introduction}
\label{sec: intro}

\noindent
In category theory the Grothendieck construction is a general method of generating fibrations that
generalises the semidirect product of monoids (see~\cite{BW12}, section 12.2). There are various instances 
in which the Grothendieck construction ``appears'' in connection to Martin-L\"of Type Theory $(\MLTT)$.
For example, Hofmann in~\cite{Ho97}, p.~140, defined the comprehension of a family over a groupoid as 
a special case
of the Grothendieck construction. In book-HoTT~\cite{HoTT13}, section 6.12, it is mentioned that the 
Sigma-type $\sum_{x : W}P(x)$, where
$W$ is a higher inductive type and $P$ is a type-family over $W$, is, from a category-theoretic point
of view, the 
``Grothendieck construction'' of $P$. In~\cite{Pa16} Palmgren used the Grothendieck construction and 
the fact that
this construction can appropriately be iterated, in order to naturally model dependent type theory in
the form 
of contextual category. Quite earlier, in~\cite{Ob89} Obtu\l owicz had also applied an iterated version 
of the 
Grothendieck construction to a hierarchy of indexed categories.  
The Grothendieck construction was also used by Hyland and Pitts in~\cite{HP89}, pp.~182-184, in their
development of categorical models of the calculus of constructions. For the relation of the Grothendieck
construction to the Chu construction see~\cite{Pe21c}.

Here we are concerned with two variants of the Grothendieck construction: the 
Grothendieck construction on $\Set$-valued presheaves and the Grothendieck construction on $\Cat$-valued
presheaves.
The first, also known as the category of elements, was ``first done by Yoneda 
and developed by Mac Lane well before Grothendieck'' (see~\cite{MM92}, p.~44), and it is used e.g., 
in the proof of the fact that every $\Set$-valued presheaf is a colimit of representable presheaves 
(see~\cite{MM92}, pp.~42-43). The second, is the original construction 
of Grothendieck (see~\cite{Gr71a}).

Our main observation is that certain equivalences in $\MLTT$ that involve the Sigma-type
motivate isomorphisms between corresponding categories that involve the Grothendieck construction.
The type-theoretic axiom of choice (Theorem~\ref{thm: ac}) and 
the ``associativity'' of the Sigma-type (Theorem~\ref{thm: assoc}) are the main examples of
this phenomenon that are treated here. 
The proof of Theorem~\ref{thm: ac} can also be seen as the translation the proof of the type-theoretic
axiom of choice.
Actually, our proof is closer to the translation of the type-theoretic proof in Bishop Set Theory $(\BST)$ 
(see~\cite{Pe19} and~\cite{Pe20}), and it requires the use both of $\Set$-valued presheaves
and $\Cat$-valued presheaves.

The aforementioned phenomenon is rooted to the fact that quite often the Grothendieck construction 
has a behavior in category theory analogous to that of the Sigma-type in $\MLTT$, or to that of the
disjoint union 
of a set-indexed family of (Bishop) sets in ($\BST$) set theory (see~\cite{Pe19},~\cite{Pe20} for 
the Bishop case).
This is clear for the definition of the 
objects of the constructed category. 
The ``interpretation'' of the Grothendieck construction as the categorical version of the disjoint union
of sets is 
justified, for example, by the role of the Grothendieck construction in the proof of equivalence between 
the two ways of describing 
families of categories; the \textit{pointwise indexing} and the \textit{display indexing} given by a
fibration.
It is exactly the role played by the disjoint union of a family of sets in the proof of equivalence
between the 
two corresponding ways of describing families of sets. Both, the disjoint union and the Grothendieck
construction, 
are used in the proof of getting a display indexing from a pointwise indexing (see~\cite{Ja99},
pp.~20-21 and p.~111). In~\cite{Ja99}, p.~29, Jacobs mentions that this switching between the two
representations
of families of categories through the Grothendieck construction ``is an extension of what we have
for sets''. What
our analysis here shows is that this is not all an accident.


%

\section{The product set of a $\Set$-valued presheaf}
\label{sec: GCset}

Throughout this paper $\C C, \C D$ are small categories, $\Fun(\C C, \C D)$ is the (small) category 
of functors from
$\C C$ to $\C D$, $\Set$ is the category of sets, $\PSh(\C C) = \Fun(\C C^{\op}, \Set)$ is the category
of $\Set$-valued
presheaves on $\C C$ (or contravariant functors from $\C C$ to $\Set$), $R \colon (\C C \times \C D)^{\op}
\to \Set$ 
is a $\Set$-valued presheaf on the (small) category $\C C \times \C D$, and $\Cat$ is the category of 
small categories.
For all categorical notions not defined here we refer to \cite{Aw10},~\cite{BW12}. 
We do not include the proofs of facts that are straightforward to show.

%

\begin{definition}[The Yoneda-Mac Lane construction]\label{def: GCset}
If $P = (P_0, P_1) \colon \C C^{\op} \to \Set \in \PSh(\C C)$, the \textit{category of elements} 
$\Sigma(\C C, P)$ of $P$ has objects pairs $(a, u)$, where $a \in \Ob_{\C C}$ and $u \in P_0(a)$. 
We denote the
disjoint union of the sets $P_0(a)$, where $a \in \Ob_{\C C}$, by 
$$\sum_{a \in \Ob_{\C C}}P_0(a).$$
A morphism $f^* \colon (a, u) \to (b, v)$ is a morphism $f \colon a \to b$ such that $[P_1(f)](v) = u$. 
If $g^* \colon (b, v)
\to (c, w)$, then $g^* \circ f^* = g \circ f$, and $1_{(a, u)} = 1_a$. 
\end{definition}

More standard notations for the category of elements are  
$$\int(\C C,P), \ \ \ \int_{\C C} P,$$
but here we follow Palmgren's notation used in~\cite{Pa16} for it. Since $\C C$ is small, the 
category $\Sigma(\C C, P)$ 
is also small. An implementation of the category of elements in $\MLTT$ 
would treat $\Ob_{\C C}$ as a type in some universe of types 
$\C U$, $P_0$ as a type family $P_0 \colon \Ob_{\C C} \to \C U$ over $\Ob_{\C C}$, and the objects of 
$\Sigma(\C C, P)$ as the type
$$\sum_{a : \Ob_{\C C}}P_0(a).$$
The connection of the category of elements with the $\Sigma$-type of $\MLTT$ fully justifies 
Palmgren's notation. 
It is immediate to see that $\pr_1^{P} \colon \Sigma(\C C, P) \to \C C$, where 
$\pr_1^{P} = \big((\pr_1^{P})_0, (\pr_1^{P})_1\big)$ with $(\pr_1^P)_0(a, u) = a$ 
and $(\pr_1^P)_1(f^*) = f$, is a functor. Actually, $\pr_1^P$ is a discrete fibration.
If $a \in \Ob_{\C C}$, the functor $\C Y^a \colon \C C^{\op} \to \Set$ is defined by $\C Y^a_0(b) = 
\Mor^{\C C}(b, a)$
and  if $f \colon b \to c$ in $\C C$, then $\C Y^a_1(f) \colon \Mor^{\C C}(c, a) \to \Mor^{\C C}(b, a)$
is defined by 
the rule $h \mapsto h \circ f$, for every $h \in \Mor^{\C C}(c, a)$.  
It is immediate to see that $\Sigma(\C C, \C Y^a)$ is the slice category $\C C/a$, and
it is straightforward to show  
that the slice category $\PSh(\C C)/P$ is equivalent to $\PSh(\Sigma(\C C, P))$.
%
%
%
%
%
%

%
%


\begin{proposition}\label{prp: sigma1}
We have that $\Sigma(\C C, -)  = (\Sigma(\C C, -)_0, \Sigma(\C C, -)_1) \colon \PSh(\C C) \to \Cat$, where
$$\Sigma(\C C, -) = (\Sigma(\C C, -)_0, \Sigma(\C C, -)_1),$$ 
$$\Sigma(\C C, -)_0(P) = \Sigma(\C C, P); \ \ \ \ P \in \PSh(\C C),$$
$$\Sigma(\C C, -)_1(\eta \colon P \To Q) \colon \Sigma(\C C, P) \to \Sigma(\C C, Q)$$
$$[\Sigma(\C C, -)_1(\eta)]_0 \colon \sum_{a \in \Ob_{\C C}}P_0(a) \to \sum_{a \in \Ob_{\C C}}Q_0(a)$$
$$[\Sigma(\C C, -)_1(\eta)]_0 (a, u) = (a, \eta_a(u)); \ \ \ \ (a, u) \in \sum_{a \in \Ob_{\C C}}P_0(a).$$
Moreover, if $f^* \colon (a, u) \to (b, v)$, then 
$[\Sigma(\C C, -)_1(\eta)]_1(f^*) = f$.
\end{proposition}

\begin{proof}
We only show that $[\Sigma(\C C, -)_1(\eta)]_1(f^*) \colon (a, \eta_a(u)) \to (b, \eta_b(v))$, as 
by the commutativity
of the corresponding diagram we get $[Q_1(f)](\eta_b(v)) = \eta_a\big([P_1(f)](v) = \eta_a(u)$.
\end{proof}

%


\begin{proposition}\label{prp: ac1}
If $F \in \Fun(\C C, \C D)$, then $R^F = (R^F_0, R^F_1) \colon \C C^{\op} \to \Set$, where
$$R^F_0(a) = R_0(a, F_0(a)); \ \ \ \ a \in \Ob_{\C C},$$
$$R^F_1(f \colon a \to b) \colon R_0(b, F_0(b)) \to R_0(a, F_0(a))$$
$$R^F_1(f) = R_1(f, F_1(f)),$$
as $(f, F_1(f)) \colon (a, F_0(a)) \to (b, F_0(b))$ in $\C C \times \C D$. 
\end{proposition}

\begin{definition}\label{def: productset}
If $P \colon \C C^{\op} \to \Set$ is a $\Set$-valued presheaf on $\C C$, the \textit{product set} 
$\prod_{a \in \Ob_{\C C}}P_0(a)$ of $P$ has elements families $\Phi = (\Phi_a)_{a \in \Ob_{\C C}}$, where 
$\Phi_a \in P_0(a)$, for every $a \in \Ob_{\C C}$, such that 
$$\forall_{a, b \in \Ob_{\C C}}\forall_{f \in \Mor^{\C C}(a, b)}\big([P_1(f)](\Phi_b) = \Phi_a\big).$$
\end{definition}

\begin{proposition}\label{prp: ac2}
We have that  $\Pi R = ((\Pi R)_0, (\Pi R)_1) \colon \Fun(\C C, \C D)^{\op} \to \Set$, where
$$\big(\Pi R\big)_0(F) = \prod_{a \in \Ob_{\C C}}R^F_0(a) = \prod_{a \in \Ob_{\C C}}R_0(a, F_0(a)); \ \ \ \ F \in
\Fun(\C C, \C D),$$
$$\big(\Pi R\big)_1(\eta \colon F \To G) \colon \prod_{a \in \Ob_{\C C}}R_0(a, G_0(a)) \to
\prod_{a \in \Ob_{\C C}}R_0(a, F_0(a))$$
$$\bigg[\big[\big(\Pi R\big)_1(\eta)\big](\Phi)\bigg]_a = [R_1(1_a, \eta_a)](\Phi_a); \ \ \ \ \Phi \in 
\prod_{a \in \Ob_{\C C}}R_0(a, G_0(a)), \ a \in \Ob_{\C C}.$$

\end{proposition}

\begin{proof}
We show that $\big(\Pi R\big)_1(\eta)$ is well-defined i.e., $\big(\Pi R\big)_1(\eta)\big](\Phi) \in 
\prod_{a \in \Ob_{\C C}}R_0(a, F_0(a))$, if $\Phi \in \prod_{a \in \Ob_{\C C}}R_0(a, G_0(a))$. If 
$f \colon a \to b$ in $\C C$, then by the definition of $R^F$ in Proposition~\ref{prp: ac1}, and by
Definition~\ref{def: productset} we have that
$$[R^G_1(f)](\Phi_b) = [R_1(f, G_1(f))](\Phi_b) = \Phi_a.$$
We need to show that 
$$[R^F_1(f)]\bigg(\bigg[\big(\Pi R\big)_1(\eta)\big](\Phi)\bigg]_b\bigg) = 
\bigg[\big(\Pi R\big)_1(\eta)\big](\Phi)\bigg]_a \ \ \ \ \ \mbox{i.e.,}$$
$$[R_1(f, F_1(f))]\bigg([R_1(1_b, \eta_b)](\Phi_b)\bigg) = [R_1(1_a, \eta_a)](\Phi_a).$$
Since $\eta$ is a natural transformation we have that
\begin{align*}
\ \ \ \ \ \ \ \ [R_1(f, F_1(f))]\bigg([R_1(1_b, \eta_b)](\Phi_b)\bigg) & = [R_1(f, F_1(f))] \circ R_1(1_b, 
\eta_b)](\Phi_b)\\
& = \big[R_1\big((1_b, \eta_b) \circ (f, F_1(f))\big)\big](\Phi_b)\\
& = \big[R_1\big((f, \eta_b \circ F_1(f)\big)\big](\Phi_b)\\
& = \big[R_1\big((f, G_1(f) \circ \eta_a \big)\big](\Phi_b)\\
& = \big[R_1\big((f \circ 1_a, G_1(f) \circ \eta_a \big)\big](\Phi_b)\\
& = \big[R_1\big((f, G_1f)) \circ (1_a, \eta_a)\big)\big](\Phi_b)\\
& = [R_1(1_a, \eta_a)]\bigg(\big[R_1(f, G_1(f))\big](\Phi_b)\bigg)\\
& = [R_1(1_a, \eta_a)](\Phi_a).
\end{align*}
The rest of the proof is straightforward.
\end{proof}

Next we describe the second projection associated to the Grothendieck construction similarly to
the definition 
of the second projection associated to the $\mathsmaller{\sum}$-type.

\begin{proposition}\label{prp: pr2}
Let $P \colon \C C^{\op} \to \Set$.\\[1mm]
(i) $P^{\Sigma} = (P^{\Sigma}_0, P^{\Sigma}_1) \colon \Sigma(\C C, P)^{\op} \to Set$, where 
$P^{\Sigma}_0(a, u) = P_0(a)$, for every
$(a, u) \in \Sigma(\C C, P)$, and $P^{\Sigma}_1(f^* \colon (a, u) \to (b, v)) = P_1(f) \colon P_0(b)
\to P(a)$,
for every morphism $f^* \colon (a, u) \to (b, v)$.\\[1mm]
(ii) The family $\pr_2^P = (\pr_2^P(a, u))_{(a, u) \in \Ob_{\Sigma(\C C, P}}$, where $\pr_2^P(a, u) = u$, 
for every
$(a, u) \in \Ob_{\Sigma(\C C, P)}$, belongs to the product set
$$\prod_{(a, u) \in \Ob_{\Sigma(\C C, P)}}P^{\Sigma}_0(a, u) = \prod_{(a, u) \in \sum_{a \in 
\Ob_{\C D}}P_0(a)}P_0(a).$$

\end{proposition}

\begin{proof}
The proof of (i) is immediate, while for (ii) it suffices to show the condition of 
Definition~\ref{def: productset}. 
Let $f^* \colon (a, u) \to (b, v)$ in $\Sigma(\C C, P)$ i.e., $f \colon a \to b$ such that 
$[P_1(f)](v) = u$. Thus
$\big[P_1^{\Sigma}(f^*)\big]\big(\pr_2^P(b, v)\big) = [P_1(f)](v) = u = \pr_2^P(a, u)$.
\end{proof}

\section{The distributivity of $\Pi$ over $\Sigma$ and the Grothendieck construction}
\label{sec: ac}

%
%

We shall use calligraphic letters for $\Cat$-valued presheaves.
For a covariant version of the following construction see~\cite{BW12}, pp.~337-338. 

\begin{definition}[The Grothendieck construction]\label{def: GCcat}
If $\C P = (\C P_0, \C P_1) \colon \C C^{\op} \to \Cat$,
the category $\Sigma(\C C, \C P)$
has objects pairs $(a, x)$, where $a \in \Ob_{\C C}$ and $x \in \Ob_{\C P_0(a)}$.
We denote $\Ob_{\Sigma(\C C, \C P)}$, the disjoint union of the sets $\Ob_{\C P_0(a)}$, 
where $a \in \Ob_{\C C}$, by 
$$\sum_{a \in \Ob_{\C C}}\Ob_{\C P_0(a)}.$$
A morphism from $(a, x)$ to
$(b, y)$ is a pair $(f, \phi)$, where $f \colon a \to b$ in $\C C$ and $\phi \colon x \to 
[\C P_1(f)]_0(y)$ in $\C P_0(a)$.
If $(g, \theta) \colon (b, y) \to (c, z)$, the composition
$(g, \theta) \circ (f, \phi) \colon (a, x) \to (c, z)$ is the pair $\big(g \circ f, 
[\C P_1(f)]_1(\theta) \circ \phi\big)$, 
where 
$[\C P_1(f)]_1(\theta) \circ \phi \colon x \to [\C P_1(g \circ f)]_0(z) = 
[\C P_1(f)]_0\big([\C P_1(g)]_0(z)\big)$. 
Finally,  $1_{(a, x)} = (1_a, 1_x)$. 
\end{definition}

If we consider a set as a discrete category, then Definition~\ref{def: GCcat} is a generalisation of 
Definition~\ref{def: GCset}. 
Let $\C P^{\C D} \colon \C C^{\op} \to \Cat$ the constant presheaf $\C D$ on $\C C$ i.e., 
$\C P^{\C D}_0(a) = \C D$, 
for every $a \in \Ob_{\C C}$ and $\C P^{\C D}_1(f \colon a \to b) = 1_{\C D}$, for every
$f \in \Mor^{\C C}(a, b)$. 
It is immediate to see that 
$$\Sigma(\C C, \C P^{\C D}) = \C C \times \C D.$$
This is the translation of the type-theoretic equality 
$$\sum_{x : A}B \equiv A \times B.$$
Clearly, $\pr_1^{\C P} \colon \Sigma(\C C, \C P) \to \C C$, where 
$\pr_1^{\C P} = \big((\pr_1^{\C P})_0, (\pr_1^{\C P})_1\big)$ with $(\pr_1^{\C P})_0(a, x) = a$ 
and $(\pr_1^{\C P})_1\big((f, \phi) \colon (a, x) \to (b, y)\big) = f$, is a functor. Actually, 
$\pr_1^{\C P}$ is
a split fibration. Next we translate accordingly the type-theoretic equivalence
$$\sum_{x : A}\sum_{y : B} R(x, y) \simeq \sum_{y : B}\sum_{x : A} R(x, y).$$

\begin{proposition}\label{prp: presheafofcats}
Let $a \in \Ob_{\C C}$.\\[1mm]
(i) $R^a = (R^a_0, R^a_1) \colon \C D^{\op} \to \Set$, where $R^a_0(x) = R_0(a, x)$, 
for every $x \in \Ob_{\C D}$, and $R^a_1(\phi \colon x \to y) = R_1(1_a, \phi) \colon R_0(a, y) 
\to R_0(a, x)$,
for every $\phi \colon x \to y$ in $\C D$.\\[1mm]
(ii) $^xR = (^xR_0, ^xR_1) \colon \C C^{\op} \to \Set$, where $^xR_0(a) = R_0(a, x)$, for every 
$a \in \Ob_{\C C}$, 
and $^xR_1(f \colon a \to b) = R_1(f, 1_x) \colon R_0(b, x) \to R_0(a, x)$, for every $f \colon 
a \to b$ in $\C C$.\\[1mm]
(iii) $\Sigma^{\C D, R} = (\Sigma^{\C D, R}_0, \Sigma^{\C D, R}_1) \colon \C C^{\op} \to \Cat$, where 
$$\Sigma^{\C D, R}_0(a) = \Sigma(\C D, R^a); \ \ \ \ a \in \Ob_{\C C},$$
$$\Sigma^{\C D, R}_1(f) \colon \Sigma(\C D, R^b) \to \Sigma(\C D, R^a); \ \ \ \ f \in 
\Mor^{\C C}(a, b),$$
$$\big[\Sigma^{\C D, R}_1(f)\big]_0 \colon \sum_{x \in \Ob_{\C D}}R_0(b, x) \to 
\sum_{x \in \Ob_{\C D}}R_0(a, x) $$
$$\big[\Sigma^{\C D, R}_1(f)\big]_0(x, u) = \big(x, [R_1(f, 1_x)](u)\big),$$
$$\big[\Sigma^{\C D, R}_1(f)\big]_1\big(\phi^* \colon (x, u) \to (y, v)\big) \colon
\big(x, [R_1(f, 1_x)](u)\big)
\to \big(y, [R_1(f, 1_y)](v)\big),$$
$$\big[\Sigma^{\C D, R}_1(f)\big]_1\big(\phi^*\big) = \phi.$$
(iv) $\Sigma^{\C C, R} = (\Sigma^{\C C, R}_0, \Sigma^{\C C, R}_1) \colon \C D^{\op} \to \Cat$, where 
$$\Sigma^{\C C, R}_0(x) = \Sigma(\C C, ^xR); \ \ \ \ x \in \Ob_{\C D},$$
$$\Sigma^{\C C, R}_1(\phi) \colon \Sigma(\C C, ^yR) \to \Sigma(\C C, ^xR); \ \ \ \ \phi \in 
\Mor^{\C D}(x, y),$$
$$\big[\Sigma^{\C C, R}_1(\phi)\big]_0 \colon \sum_{x \in \Ob_{\C C}}R_0(a, y) \to 
\sum_{x \in \Ob_{\C C}}R_0(a, x) $$
$$\big[\Sigma^{\C C, R}_1(\phi)\big]_0(a, u) = \big(a, [R_1(1_a, \phi)](u)\big),$$
$$\big[\Sigma^{\C C, R}_1(\phi)\big]_1\big(f^* \colon (a, u) \to (b, v)\big) \colon
\big(a, [R_1(1_a, \phi)](u)\big)
\to \big(b, [R_1(1_b, \phi)](v)\big),$$
$$\big[\Sigma^{\C C, R}_1(\phi)\big]_1\big(f^*\big) = f.$$
(v) The categories $\Sigma(\C C, \Sigma^{\C D, R})$ and $\Sigma(\C D, \Sigma^{\C C, R})$ are isomorphic.
\end{proposition}

\begin{proof}
The proofs of (i) and (ii) are immediate. For (iii) we only show that 
$\big[\Sigma^{\C D, R}_1(f)\big]_1$ is well-defined. 
If $\phi^* \colon (x, u) \to (y, v)$ in $\Sigma(\C D, R^b)$ i.e., $\phi \colon x \to y$ 
such that $[R_1(1_b, \phi)](v) = u$,
we show that $\phi \colon \big(x, [R_1(f, 1_x)](u)\big) \to \big(y, [R_1(f, 1_y)](v)\big)$ 
in $\Sigma(\C D, R^a)$, as
\begin{align*}
\ \ \ \ \ \ \ \ \ \ \ \ \  [R_1(1_a, \phi)]\big([R_1(f, 1_y)](v)\big) & = [R_1(1_a, \phi) 
\circ R_1(f, 1_y)](v)\\
& = \big[R_1\big((f, 1_y) \circ (1_a, \phi)\big)\big](v)\\
& = [R_1(f \circ 1_a, 1_y \circ f)](v)\\
& = [R_1(f, \phi)](v)\\
& = [R_1(1_b \circ f, \phi \circ 1_x)](v)\\
& = \big[R_1\big((1_b, \phi) \circ (f, 1_x)\big)\big](v)\\
& = [R_1(f, 1_x)]\big([R_1(1_b, \phi)](v)\big)\\
& = [R_1(f, 1_x)](u).
\end{align*}
The proof of (iv) is similar to the proof of (iii), and the proof of 
(v) is straightforward.
%
\end{proof}

Next we extend Definition~\ref{def: productset} to $\Cat$-valued presheaves.

\begin{definition}\label{def: Productset}
If $\C P \colon \C C^{\op} \to \Cat$, the \textit{product set} 
$\prod_{a \in \Ob_{\C C}}\C P_0(a)$ of $\C P$ has elements families $\Phi = (\Phi_a)_{a \in \Ob_{\C C}}$,
where
$\Phi_a \in \Ob_{\C P_0(a)}$, for every $a \in \Ob_{\C C}$, such that 
for every $a, b \in \Ob_{\C C}$ and for every $f \in \Mor^{\C C}(a, b)$ there is a morphism 
$\phi \colon \Phi_a 
\to [\C P_1(f)]_0(\Phi_b)$ in $\C P_0(a)$.
\end{definition}

If $\Phi$ is in the product set of the presheaf $\Sigma^{\C D, R} \colon \C C^{\op} \to \Cat$ i.e., 
$$\Phi \in \prod_{a \in \Ob_{\C C}}\sum_{x \in \Ob_{\C D}}R_0(a, x),$$
and if $\Phi_a = (x, u)$ with $x \in \Ob_{\C D}$ and $u \in R_0(a, x)$ and if $\Phi_b = (y, v)$ 
with $y \in \Ob_{\C D}$
and $v \in R_0(b, y)$, then if $f \colon a \to b$ in $\C C$, we have that
$$\big[\Sigma^{\C D, R}_1(f)\big]_0(y, v) = \big(y, [R_1(f, 1_y)](v)\big),$$
If $\phi^* \colon (x, u) \to \big(y, [R_1(f, 1_y)](v)\big)$ in $\Sigma(\C D, R^a)$, there is
$\phi \colon x \to y$ in
$\C D$ such that
\begin{align*}
\ \ \ \ \ \ \ \ \ \ \ \ \ \ \ \ \ \ \ \ \ \ \ \ \ \ \ \ \ \ \ \ \ u & = [R_1(1_a, 
\phi)]\big([R_1(f, 1_y)](v)\big)\\
& = [R_1(1_a, \phi) \circ R_1(f, 1_y)](v)\\
& = \big[R_1 \big((f, 1_y) \circ (1_a, \phi)\big)\big](v)\\
& = [R_1(f \circ 1_a, 1_y \circ \phi)](v)\\
& = [R_1(f, \phi)](v).
\end{align*}

Since $R^a \colon \C D^{\op} \to \Set$, we have that 
$$\pr_2^{R^a} \in \prod_{(x,u) \in \sum_{x \in \Ob_{\C D}}R_0(a, x)}R_0(a, x),$$
with $\pr_2^{R^a}(x, u) = u$. The above equality $u = [R_1(f, \phi)](v)$ is thus written as
$$\pr_2^{R^a}(\Phi_a) = [R_1(f, \phi)]\big(\pr_2^{R^b}(\Phi_b)\big).$$

Next we add a form of ``realiser'' for the defining condition of the product set of $\Sigma^{\C D,R}$.
In the case of the type-theoretic axiom of choice, the corresponding function from $X$ to $Y$ 
is definable. This does not seem possible for the categories used here.

\begin{definition}\label{def: associate}
If $\Phi$ is in the product set of $\Sigma^{\C D,R}$, an \textit{associate} for $\Phi$ 
is a functor
$F^{\Phi} \colon \C C \to \C D$, such that for every $a, b \in \Ob_{\C C}$ and for every
morphism $f \colon a \to b$ in 
$\C C$ we have that
$$\big(F^{\Phi}\big)_0(a) = \pr_1^{R^a}(\Phi_a),$$
$$\big(F^{\Phi}\big)_1(f) \colon \Phi_a \to \big[\Sigma^{\C D, R}_1(f)\big]_0(\Phi_b)$$
i.e., the morphism $\big(F^{\Phi}\big)_1(f) \colon \pr_1^{R^a}(\Phi_a) \to \pr_1^{R^b}(\Phi_b)$ 
satisfies the defining
condition of a morphism $\Phi_a \to \big[\Sigma^{\C D, R}_1(f)\big]_0(\Phi_b)$ in $\Sigma(\C D, R^a)$.
\end{definition}

\begin{definition}\label{def: newcat}
The \textit{product category} $\Pi (\C C, \Sigma^{\C D, R})$ of $\C C, \C D$ with respect
to $R$ has objects pairs 
$(\Phi, F^{\Phi})$, where $\Phi = (\Phi_a)_{a \in \Ob_{\C C}}$ is in the product set of the $\Cat$-valued
presheaf 
$\Sigma^{\C D,R}$ and $F^{\Phi}$ is an associate for $\Phi$. We denote the objects of this category by 
$$\bigg[\prod_{a \in \Ob_{\C C}}\sum_{x \in \Ob_{\C D}}R_0(a, x)\bigg]^*.$$
A morphism from $(\Phi, F^{\Phi})$ to $(\Psi, F^{\Psi})$ 
in $\Pi (\C C, \Sigma^{\C D, R})$ is a natural transformation $\eta \colon F^{\Phi} \To F^{\Psi}$ 
such that the following
compatibility condition between $\eta$ and $R$ is satisfied:
$$\forall_{a \in \Ob_{\C C}}\bigg([R_1(1_a, \eta_a)]\big(\pr_2^{R^a}(\Psi_a)\big) = 
\pr_2^{R^a}(\Phi_a)\bigg).$$
Moreover, $1_{(\Phi, F^{\Phi})} = 1_{F^{\Phi}}$, where $\big(1_{F^{\Phi}}\big)_a = 
1_{\pr_1^{R^a}(\Phi_a)}$, for every
$a \in \Ob_{\C C}$. The composition of morphisms in $\Pi (\C C, \Sigma^{\C D, R})$ is the 
composition of the corresponding
natural transformations.
\end{definition}

Notice that as $(1_a, \eta_a) \colon (a, F_0^{\Phi}(a)) \to (a, F_0^{\Psi}(a))$, we get $R_1(1_a, \eta_a)
\colon R_0(a, F_0^{\Psi}(a)) \to R_0(a, F_0^{\Phi}(a))$. Let $\Phi_a = (x, u)$, where
$x = F_0^{\Phi}(a) \in \Ob_{\C D}$
and $u \in R_0(a, x)$, 
$\Phi_b = (y, v)$, where $y = F_0^{\Phi}(b) \in \Ob_{\C D}$ and $v \in R_0(b, y)$, $\Psi_a = 
(x{'}, u{'})$, 
where $x{'} = F_0^{\Psi}(a) \in \Ob_{\C D}$ and $u{'} \in R_0(a, x{'})$, and $\Psi_b = (y{'}, v{'})$, 
where 
$y{'} = F_0^{\Psi}(b) \in \Ob_{\C D}$ and $v{'} \in R_0(b, y{'})$. Then the compatibility condition between 
$\eta$ and $R$ takes the form $[R_1(1_a, \eta_a)](u{'}) = u$. The identity morphism $1_{(\Phi, F^{\Phi})}$
is
well-defined, as $R_1(1_a, 1_x) = \id_{R_0(a, x)}$, hence $[R_1(1_a, 1_x)](v) = v$. The composition
of morphisms 
in $\Pi (\C C, \Sigma^{\C D, R})$ is also well-defined; if $\eta \colon F^{\Phi} \To F^X$ and 
$\theta \colon F^X \To F^{\Psi}$,
then
\begin{align*}
\ \ \ \ \ \ \ \  \ \ R_1(1_a, \theta_a \circ \eta_a)]\big(\pr_2^{R^a}(\Psi_a)\big) & = 
[R_1(1_a, \eta_a)]\big([R_1(1_a,
\theta_a)](\pr_2^{R^a}(\Psi_a)\big)\\
& = [R_1(1_a, \eta_a)]\big(\pr_2^{R^a}(X_a)\\
& = \pr_2^{R^a}(\Phi_a).
\end{align*}

\begin{theorem}\label{thm: ac}
The categories $\Pi (\C C, \Sigma^{\C D, R})$ and $\Sigma(\Fun(\C C, \C D), \Pi R)$ are isomorphic.
\end{theorem}

\begin{proof}
First we define the functor $\AC \colon \Pi (\C C, \Sigma^{\C D, R}) \to \Sigma(\Fun(\C C, \C D), \Pi R)$.
Let
$$\AC_0 \colon \bigg[\prod_{a \in \Ob_{\C C}}\sum_{x \in \Ob_{\C D}}R_0(a, x)\bigg]^* \to
\sum_{F \in \Fun(\C C, \C D)}\prod_{a \in \Ob_{\C C}}R_0(a, F_0(a))$$
$$\AC_0(\Phi, F^{\Phi}) = (F^{\Phi}, \Phi^*); \ \ \ \ (\Phi, F^{\Phi}) \in 
\bigg[\prod_{a \in \Ob_{\C C}}\sum_{x \in \Ob_{\C D}}R_0(a, x)\bigg]^*$$
$$\Phi_a^* = \pr_2^{R^a}(\Phi_a); \ \ \ \ a \in \Ob_{\C C}.$$
First we show that $\AC_0$ is well-defined i.e., 
$$\Phi^* \in \prod_{a \in \Ob_{\C C}}R_0^F(a) =  \prod_{a \in \Ob_{\C C}}R_0(a, F_0^{\Phi}(a)) =
\prod_{a \in \Ob_{\C C}}R_0(a, \pr_1^{R^a}(\Phi_a)).$$
By Definition~\ref{def: productset} it suffices to show that if $f \colon a \to b$ in $\C C$, then 
$\Phi^*_a = \big[R_1^F(f)\big](\Phi_b^*)$ i.e., 
$$\pr_2^{R^a}(\Phi_a) = \big[R_1(f, F_1^{\Phi}(f))\big](\pr_2^{R^b}(\Phi_b)\big).$$
As we have already explained right after Definition~\ref{def: Productset}, this follows from the hypothesis 
$$\Phi \in \prod_{a \in \Ob_{\C C}}\sum_{x \in \Ob_{\C D}}R_0(a, x).$$
If $\eta \colon F^{\Phi} 
\To F^{\Psi}$ is a 
morphism from $(\Phi, F^{\Phi})$ to $(\Psi, F^{\Psi})$, let
$$\AC_1(\eta) \colon (F^{\Phi}, \Phi^*) \to (F^{\Psi}, \Psi^*) $$
$$\AC_1(\eta) = \eta.$$
We show that $\eta$ is also a morphism in $\Sigma(\Fun(\C C, \C D), \Pi R)$ i.e., 
$$\big[(\Pi R)_1(\eta)\big](\Psi^*) = \Phi^*.$$
If $a \in \Ob_{\C C}$, then by 
the compatibility condition between $\eta$ and $R$ we have that 
$$\bigg[\big[\big(\Pi R\big)_1(\eta)\big](\Psi^*)\bigg]_a = [R_1(1_a, \eta_a)](\Psi_a^*) = 
[R_1(1_a, \eta_a)]\big(\pr_2^{R^a}(\Psi_a)\big) =  \pr_2^{R^a}(\Phi_a) = \Phi_a^*.$$
$\AC$ is a functor. 
Next we define the functor $\CA \colon \Sigma(\Fun(\C C, \C D), \Pi R) \to \Pi (\C C, \Sigma^{\C D, R})$ by
$$\CA_0 \colon \sum_{F \in \Fun(\C C, \C D)}\prod_{a \in \Ob_{\C C}}R_0(a, F_0(a)) \to 
\bigg[\prod_{a \in \Ob_{\C C}}\sum_{x \in \Ob_{\C D}}R_0(a, x)\bigg]^*$$
$$\CA_0(F, \Phi^*) = (\Phi, F); \ \ \ \ (F, \Phi^*) \in \sum_{F \in \Fun(\C C, 
\C D)}\prod_{a \in \Ob_{\C C}}R_0(a, F_0(a)),$$
$$\Phi_a = (F_0(a), \Phi_a^*); \ \ \ \ a \in \Ob_{\C C}.$$
First we show that $\CA_0$ is well-defined i.e., $(\Phi, F) \in \Ob_{\Pi (\C C, \Sigma^{\C D, R})}$, 
which means that
$$\Phi \in \prod_{a \in \Ob_{\C C}}\sum_{x \in \Ob_{\C D}}R_0(a, x)$$
and $F$ is an associate for $\Phi$. As $F_0(a) \in \Ob_{\C D}$ and $\Phi_a^* \in R_0(a, F_0(a))$, we get 
$\Phi_a \in \sum_{x \in \Ob_{\C D}}R_0(a, x)$. Clearly, $F_0(a) = \pr_1^{R^a}(\Phi_a)$, for every
$a \in \Ob_{\C C}$.
Let $f \colon a \to b$ in $\C C$. We show that $F_1(f) \colon F_0(a) \to F_0(b)$ satisfies the 
defining condition of morphism 
$\Phi_a \to \big[\Sigma^{\C D, R}_1(f)\big]_0(\Phi_b)$ in $\Sigma(\C D, R^a)$. As $\Phi^* \in 
\prod_{a \in \Ob_{\C C}}R_0(a, F_0(a))$, we have that $[R_1(f, F_1(f))](\Phi_b^*) = \Phi_a^*$. Moreover, 
$$ \big[\Sigma^{\C D, R}_1(f)\big]_0(\Phi_b) = \big[\Sigma^{\C D, R}_1(f)\big]_0(F_0(b), \Phi_b^*)
= \big(F_0(b), [R_1(f, 1_{F_0(b)})](\Phi_b^*)\big).$$
Hence,
\begin{align*}
\ \ \ \ \ \ \ \ \ \ [R_1(1_a, F_1(f))]\big([R_1(f, 1_{F_0(b)})](\Phi_b^*)\big) & =
\big[R_1\big((1_a, F_1(f))
\circ (1_a, F_1(f))\big)\big](\Phi_b^*)\\
& = [R_1(f \circ 1_a, 1_{F_0(b)} \circ F_1(f))](\Phi_b^*)\\
& = [R_1(f, F_1(f))](\Phi_b^*)\\
& = \Phi_a^*.
\end{align*}
If $\eta \colon (F, \Phi^*) \to (G, \Theta^*)$ in $\Sigma(\Fun(\C C, \C D), \Pi R)$, let 
$\CA_1(\eta) \colon (\Phi, F) \to (\Theta, G)$, defined by the rule 
$\CA_1(\eta) = \eta$.
We show that $\eta$ is also a morphism in $\Pi (\C C, \Sigma^{\C D, R})$ i.e., $\eta$ satisfies 
the compatibility 
condition with $R$. As $\eta$ is a morphism in $\Sigma(\Fun(\C C, \C D), \Pi R)$, we have that 
$$\big[(\Pi R)_1(\eta)\big](\Theta^*) = \Phi^*,$$
hence 
$$\bigg[\big[\big(\Pi R\big)_1(\eta)\big](\Theta^*)\bigg]_a = [R_1(1_a, \eta_a)](\Theta_a^*) = \Phi_a^*,$$
for every $a \in \Ob_{\C C}$. Thus, for every $a \in \Ob_{\C C}$ we get
$$[R_1(1_a, \eta_a)]\big(\pr_2^{R^a}(\Theta_a)\big) = [R_1(1_a, \eta_a)](\Theta_a^*) = \Phi_a^* = 
\pr_2^{R^a}(\Phi_a).$$
Since $\AC_0\big(\CA_0(F, \Phi^*)\big) = \AC_0(\Phi, F) = (F, \Phi^*)$ and $\CA_0\big(\AC_0(\Phi, 
F^{\Phi})\big)
= \CA_0(F^{\Phi}, \Phi^*) = (\Phi, F^{\Phi})$, the two categories are isomorphic. 
\end{proof}

This result is the category-theoretic analogue to the type-theoretic axiom, and the equivalence 
between the types 
involved, for the Grothendieck construction.
It also shows that the product category $\Pi (\C C, \Sigma^{\C D, R})$ is non-trivial, as it is
essentially the easier to describe category $\Sigma(\Fun(\C C, \C D), \Pi R)$.

\section{The ``associativity'' of the Grothendieck construction}
\label{sec: assoc}

The ``associativity'' of the $\sum$-type is the following equivalence (see Ex.~2.10 in~\cite{HoTT13}): 
$$\sum_{x : A} \sum_{y : B(x)} C(x,y) \simeq \sum_{p : \sum_{x : A}B(x)}C(p),$$
where $A : \C U$, $B \colon A \to \C U$, and $C \colon \big(\sum_{x : A}B(x)\big) \to \C U$.
Next we translate this property of the $\Sigma$-type to the Grothendieck construction by lifting 
Proposition~\ref{prp: presheafofcats} one level up, as our starting presheaves are $\Cat$-valued.

\begin{theorem}\label{thm: assoc}
Let $\C P \colon \C C^{\op} \to \Cat$, $a \in \Ob_{\C C}$ and $\C Q \colon \Sigma(\C C, \C P)^{\op}
\to \Cat$.\\[1mm]
\normalfont (i)
\itshape $\C Q^a = (\C Q_0^a, \C Q_1^a) \colon \C P_0(a)^{\op} \to \Cat$, where $\C Q^a_0(x) = 
\C Q_0(a, x)$, 
for every $x \in \Ob_{\C P_0(a)}$, and 
$\C Q^a_1(j \colon x \to x{'}) \colon \C Q_0(a, x{'}) \to \C Q_0(a, x)$ is 
$\C Q_1(1_a, j)$, for every
$j \in \Mor^{\C P_0(a)}(x,x{'})$.\\[1mm]
\normalfont (ii)
\itshape $\Sigma^{\C P, \C Q} = (\Sigma^{\C P, \C Q}_0, \Sigma^{\C P, \C Q}_1) \colon \C C^{\op} 
\to \Cat$, where 
$$\Sigma^{\C P, \C Q}_0(a) = \Sigma(\C P_0(a), \C Q^a); \ \ \ \ a \in \Ob_{\C C},$$
$$\Sigma^{\C P, \C Q}_1(f \colon a \to b) \colon \Sigma(\C P_0(b), \C Q^b) \to \Sigma(\C P_0(a), 
\C Q^a); \ \ \ \ 
f \in \Mor^{\C C}(a, b),$$
$$\big[\Sigma^{\C P, \C Q}_1(f)\big]_0 \colon \sum_{y \in \Ob_{\C P_0(b)}}\Ob_{\C Q_0(b, y)} \to 
\sum_{x \in \Ob_{\C P_0(a)}}\Ob_{\C Q_0(a, x)}$$
$$\big[\Sigma^{\C P, \C Q}_1(f)\big]_0(y, t) = \bigg(\underbrace{\big[\C P_1(f)\big]_0(y)}_x, 
\big[\C Q_1(f, 1_x)\big]_0(t)\bigg);
\ \ \ \ (y, t) \in \sum_{y \in \Ob_{\C P_0(b)}}\Ob_{\C Q_0(b, y)},$$
$$\big[\Sigma^{\C P, \C Q}_1(f)\big]_1\big((j, \mu) \colon (y, t) \to (y{'}, t{'})\big) \colon 
\big[\Sigma^{\C P, \C Q}_1(f)\big]_0(y, t) \to \big[\Sigma^{\C P, \C Q}_1(f)\big]_0(y{'}, t{'}),$$
$$\big[\Sigma^{\C P, \C Q}_1(f)\big]_1(j, \mu) \colon \bigg(x, 
\big[\C Q_1(f, 1_x)\big]_0(t)\bigg) \to 
\bigg(x{'}, \big[\C Q_1(f, 1_{x{'}})\big]_0(t{'})\bigg),$$
$$x = \big[\C P_1(f)\big]_0(y), \ \ \ x{'} = \big[\C P_1(f)\big]_0(y{'}),$$
$$\big[\Sigma^{\C P, \C Q}_1(f)\big]_1(j, \mu) = (j{'}, \mu{'}),$$
$$j{'} = \big[\C P_1(f)\big]_1(j), \ \ \ \mu{'} = \big[Q_1(f, 1_x)\big]_1(\mu).$$
\normalfont (iii)
\itshape The categories $\Sigma\big(\C C, \Sigma^{\C P, \C Q}\big) $ and $\Sigma\big(\Sigma(\C C, 
\C P), \C Q\big)$ are
isomorphic.

\end{theorem}

\begin{proof}
(i)  If $j \colon x \to x{'}$ in $\C P_0(a)$, we have that $(1_a, j) \colon (a, x) \to (a, x{'})$ 
in $\Sigma(\C C, \C P)$, as
$$[\C P_1(1_a)]_0(x{'}) = [1_{\C P_0(a)}]_0(x{'}) = \id_{\C Ob_{\C P_0(a)}}(x{'}) = x{'}.$$
The rest of the proof that $\C Q^a$ is a contravariant functor is straightforward.\\[1mm]
(ii) By (i) $\Sigma^{\C P, \C Q}_0(a)$ is well-defined, for every $a \in \Ob_{\C C}$, where
$$\Ob_{\Sigma^{\C P, \C Q}_0(a)} = \sum_{x \in \Ob_{\C P_0(a)}}\Ob_{\C Q_0^a(x)} = \sum_{x \in 
\Ob_{\C P_0(a)}}
\Ob_{\C Q_0(a,x)},$$
and $(i, \lambda) \colon (x, u) \to (y, v)$ in $\Sigma(\C P_0(a), \C Q^a)$ i.e., $i \colon x \to y$ in
$\C P_0(a)$ and $\lambda \colon u \to [\C Q_1(1_a, i)]_0(v)$ in $\C Q_0(a, x)$. If $f \colon a \to b$
in $\C C$, we show 
first that $\big[\Sigma^{\C P, \C Q}_1(f)\big]_0$ is well-defined. If $(y, t) \in 
\sum_{y \in \Ob_{\C P_0(b)}}\Ob_{\C Q_0(b, y)}$, we show that
$$\big[\Sigma^{\C P, \C Q}_1(f)\big]_0(y, t) = \bigg(\big[\C P_1(f)\big]_0(y), 
\big[\C Q_1(f, 1_x)\big]_0(t)\bigg)
\in \sum_{x \in \Ob_{\C P_0(a)}}\Ob_{\C Q_0(a, x)}.$$
As $\C P_1(f) \colon \C P_0(b) \to \C P_0(a)$, we get $\big[\C P_1(f)\big]_0(y) \in \Ob_{\C P_0(a)}$. 
We observe that
$(f, 1_x) \colon (a, x) \to (b, y)$ in $\Sigma(\C C, \C P)$,
where $x = \big[\C P_1(f)\big]_0(y)$.
Hence $\C Q_1(f, 1_x) \colon \C Q_0(b, y) \to \C Q_0(a, x)$, and consequently we get 
$\big[\C Q_1(f, 1_x)\big]_0(t) \in \Ob_{\C Q_0(a, x)}$.
Next we show that $\big[\Sigma^{\C P, \C Q}_1(f)\big]_1$ is well-defined.
Let $(j, \mu) \colon (y, t) \to (y{'}, t{'})$
in $\Sigma(\C P_0(b), \C Q^b)$ i.e., $j \colon y \to y{'}$ in $\C P_0(b)$ and 
$\mu \colon t \to [\C Q_1(1_b, j)]_0(t{'})$
in $\C Q_0(b, y)$. Consequently, $j{'} = [\C P_1(f)]_1(j) \colon \big[\C P_1(f)\big]_0(y) \to 
\big[\C P_1(f)\big]_0(y{'})$ i.e.,
$j{'} \colon x \to x{'}$, where $x{'} = \big[\C P_1(f)\big]_0(y{'})$. Next we define a morphism 
$$\mu{'} \colon \big[\C Q_1(f, 1_x)\big]_0(t) \to 
[\C Q_1(1_a, j{'})]_0\bigg(\big[\C Q_1(f, 1_{x{'}})\big]_0(t{'})\bigg)$$
in $\C Q_0(a, x)$. As $(f, 1_{x{'}}) \colon (a, x{'}) \to (b, y{'})$, we have that 
$\C Q_1(f, 1_{x{'}}) \colon
\C Q_0(b, y{'}) \to \C Q_0(a, x{'})$. Moreover, $(1_a, j{'}) \colon (a, x) \to (a, x{'})$ in 
$\Sigma(\C C, \C P)$, 
by the proof we gave in (i), and hence $\C Q_1(1_a, j{'}) \colon \C Q_0(a, x{'}) \to \C Q_0(a, x)$.
Hence by the definition
of composition of morphisms in $\Sigma(\C C, \C P)$ we get 
\begin{align*}
[\C Q_1(1_a, j{'})]_0\bigg(\big[\C Q_1(f, 1_{x{'}})\big]_0(t{'})\bigg) & = \bigg[\C Q_1\big((f, 1_{x{'}}) 
\circ (1_a, j{'})\big)\bigg]_0(t{'})\\
& = \bigg[\C Q_1\big(f \circ 1_a, [\C P_1(1_a)]_1(1_{x{'}}) \circ j{'}\big)\bigg]_0(t{'})\\
& = \bigg[\C Q_1\big(f, [1_{\C P_0(a)}]_1(1_{x{'}}) \circ j{'}\big)\bigg]_0(t{'})\\
& = \big[\C Q_1\big(f, 1_{x{'}} \circ j{'}\big)\big]_0(t{'})\\
& = \big[\C Q_1\big(f, j{'}\big)\big]_0(t{'}).
\end{align*}
Therefore we need to define a morphism  $\mu{'} \colon \big[\C Q_1(f, 1_x)\big]_0(t) \to
\big[\C Q_1\big(f, j{'}\big)\big](t{'})$. As $\mu \colon t \to [\C Q_1(1_b, j)]_0(t{'})$
in $\C Q_0(b, y)$ and $\C Q_1(f, 1_x) \colon \C Q_0(b, y) \to \C Q_0(a, x)$, we get  
$$\mu{'} = [\C Q_1(f, 1_x)]_1(\mu) \colon [\C Q_1(f, 1_x)]_0(t) \to
[\C Q_1(f, 1_x)]_0\bigg([\C Q_1(1_b, j)]_0(t{'})\bigg).$$
By the proof of (i) we have that $(1_b, j) \colon (b, y) \to (b, y{'})$ in $\Sigma(\C C, \C P)$. Since
$(f, 1_x) \colon (a, x) \to (b, y)$ in $\Sigma(\C C, \C P)$, by the definition of composition of 
morphisms in $\Sigma(\C C, \C P)$
we have that
\begin{align*}
[\C Q_1(f, 1_x)]_0\bigg([\C Q_1(1_b, j)]_0(t{'})\bigg) & = \bigg[\C Q_1\big((1_b, j) \circ
(f, 1_x)\big)\bigg]_0(t{'})\\
& = \bigg[\C Q_1\big(1_b \circ f, [\C P_1(f)]_1(j) \circ 1_x)\big)\bigg]_0(t{'})\\
& = \big[\C Q_1\big(f, j{'} \circ 1_x)\big)\big]_0(t{'})\\
& = \big[\C Q_1\big(f, j{'}\big)\big]_0(t{'}),
\end{align*}
and hence $\mu{'}$ is the required morphism. Next we prove that $\Sigma^{\C P, \C Q}_1(f)$ is a 
functor from 
$\Sigma(\C P_0(b), \C Q^b)$ to $\Sigma(\C P_0(a), \C Q^a)$. If $(y, t) \in \Ob_{\Sigma(\C P_0(b), 
\C Q^b)}$, then 
\begin{align*}
\ \ \ \ \ \ \ \ \ \ \big[\Sigma^{\C P, \C Q}_1(f)\big]_1\big(1_{(y, t)}\big) & = \big[\Sigma^{\C P,
\C Q}_1(f)\big]_1(1_y, 1_t) \\
& = \bigg(\big[\C P_1(f)\big]_1(1_y), \big[Q_1(f, 1_x)\big]_1(1_t)\bigg)\\
& = \bigg(\big[1_{[\C P_1(f)]_0(y)}, 1_{[Q_1(f, 1_x)]_0(t)}\bigg)\\
& = 1_{\big[\Sigma^{\C P, \C Q}_1(f)\big]_0(y, t)}.
\end{align*}
Let $(j, \mu) \colon (y, t) \to (y{'}, t{'})$ and $(k, \lambda) \colon (y{'}, t{'}) \to (y{''}, t{''})$ in 
$\Sigma(\C P_0(b), \C Q^b)$ i.e., $j \colon y \to y{'}$ in $\C P_0(b)$, $\mu \colon t \to
[\C Q_1(1_b, j)]_0(t{'})$ in
$\C Q_0(b, y)$, and $k \colon y{'} \to y{''}$ in $\C P_0(b)$, $\lambda \colon t{'} \to 
[\C Q_1(1_b, k)]_0(t{''})$ in
$\C Q_0(b, y{'})$. By the definition of composition of morphisms in $\Sigma(\C P_0(b), \C Q^b)$ 
we have that 
$(k, \lambda) \circ (j, \mu) \colon (y, t) \to (y{''}, t{''})$ with
$(k, \lambda) \circ (j, \mu) = \big(k \circ j, [\C Q_1(1_b, j)]_1(\lambda) \circ \mu\big)$. Hence
\begin{align*}
A & = \big[\Sigma^{\C P, \C Q}_1(f)\big]_1\big((k, \lambda) \circ (j, \mu)\big)\\
& = \big[\Sigma^{\C P, \C Q}_1(f)\big]_1 \bigg(k \circ j, [\C Q_1(1_b, j)]_1(\lambda) \circ \mu\bigg)\\
& = \bigg(\big[\C P_1(f)\big]_1(k \circ j), \big[Q_1(f, 1_x)\big]_1\big([\C Q_1(1_b, j)]_1(\lambda) 
\circ \mu \big)\bigg)\\
& = \bigg(\big[\C P_1(f)\big]_1(k) \circ \big[\C P_1(f)\big]_1(j), 
\big[Q_1(f, 1_x)\big]_1\big([\C Q_1(1_b, j)]_1(\lambda)\big) \circ \big[Q_1(f, 1_x)\big]_1\big(\mu \big)\bigg)\\
& = \bigg(\big[\C P_1(f)\big]_1(k) \circ \big[\C P_1(f)\big]_1(j), 
\big[Q_1(f, j{'})\big]_1(\lambda) \circ \mu{'}\bigg),
\end{align*}
since, as we have shown above, $(1_b, j) \circ (f, 1_x) = (f, j{'})$ in $\Sigma(\C C, \C P)$.
By definition
$$\big[\Sigma^{\C P, \C Q}_1(f)\big]_1(j, \mu) = \big(\big[\C P_1(f)\big]_1(j), 
\big[Q_1(f, 1_x)\big]_1(\mu)\big),$$
$$\big[\Sigma^{\C P, \C Q}_1(f)\big]_1(k, \lambda) = \big(\big[\C P_1(f)\big]_1(k),
\big[Q_1(f, 1_{x{'}})\big]_1(\lambda)\big),$$
and by the definition of composition in $\Sigma(\C P_0(a), \C Q^a)$ we have that 
\begin{align*}
B & = \big[\Sigma^{\C P, \C Q}_1(f)\big]_1(k, \lambda) \circ \big[\Sigma^{\C P, \C Q}_1(f)\big]_1(j, \mu)\\
& = \bigg(\big[\C P_1(f)\big]_1(k), \big[Q_1(f, 1_{x{'}})\big]_1(\lambda)\bigg) \circ 
\bigg(\big[\C P_1(f)\big]_1(j), \big[Q_1(f, 1_{x{'}})\big]_1(\mu)\bigg)\\
& = \bigg(\big[\C P_1(f)\big]_1(k) \circ \big[\C P_1(f)\big]_1(j), \
[\C Q_1^a(j{'})]_1\big(\big[Q_1(f, 1_{x{'}})\big]_1(\lambda)\big) \circ \mu{'}\bigg)\\
& = \bigg(\big[\C P_1(f)\big]_1(k) \circ \big[\C P_1(f)\big]_1(j), \
[\C Q_1(1_a, j{'})]_1\big(\big[Q_1(f, 1_{x{'}})\big]_1(\lambda)\big) \circ \mu{'}\bigg)\\
& = \bigg(\big[\C P_1(f)\big]_1(k) \circ \big[\C P_1(f)\big]_1(j),
\big[Q_1(f, j{'})\big]_1(\lambda) \circ \mu{'}\bigg),
\end{align*}
since, as we have shown above, $(f, 1_x) \circ (1_a, j{'}) = (f, j{'})$ in $\Sigma(\C C, \C P)$.
Consequently,
we get the required equality $A = B$. Next we show that $\Sigma^{\C P, \C Q}$ is a contravariant
functor from $\C C$ to $\Cat$. If $a \in \Ob_{\C C}$, we show that 
$\Sigma^{\C P, \C Q}_1(1_a) \colon \Sigma(\C P_0(a), \C Q^a) \to \Sigma(\C P_0(a), \C Q^a)$
is the unit-functor $1_{\Sigma(\C P_0(a), \C Q^a)}$. If $(x, s) \in \sum_{x \in 
\Ob_{\C P_0(a)}}\Ob_{\C Q_0(a, x)}$, then
\begin{align*}
\ \ \ \ \ \ \ \ \big[\Sigma^{\C P, \C Q}_1(1_a)\big]_0(x, s) & = \bigg(\big[\C P_1(1_a)\big]_0(x),
\big[\C Q_1(1_a, 1_x)\big]_0(s)\bigg)\\
& = \bigg(\big[1_{\C P_0(a)}\big]_0(x), [\C Q_1(1_{a,x)})]_0(s)\bigg)\\
& = (x, s).
\end{align*}
If $(i, \lambda) \colon (x, s) \to (x{'}, s{'})$ in $\Sigma(\C P_0(a), \C Q^a)$, then 
$$\big[\Sigma^{\C P, \C Q}_1(1_a)\big]_1(i, \lambda) = \bigg(\big[\C P_1(1_a)\big]_1(i), 
\big[Q_1(1_a, 1_x)\big]_1(\lambda)\bigg) = (i, \lambda).$$
If $g \colon b \to c$ in $\C C$, we show that 
$\big[\Sigma^{\C P, \C Q}\big]_1(g \circ f) = \big[\Sigma^{\C P, \C Q}\big]_1(f) \circ 
\big[\Sigma^{\C P, \C Q}\big]_1(g)$.
If $(z, w) \in \Sigma(\C P_0(c), \C Q^c)$, then 
$$C = \bigg[\big[\Sigma^{\C P, \C Q}\big]_1(g \circ f)\bigg]_0(z, w) = 
\bigg(\underbrace{[\C P_1(g \circ f)]_0(z)}_x, 
[\C Q_1(g \circ f, 1_x)]_0(w)\bigg),$$
where $x = [\C P_1(g \circ f)]_0(z) = [\C P_1(f) \circ \C P_1(g)]_0(z)$. Moreover,
\begin{align*}
D & = \bigg[\big[\Sigma^{\C P, \C Q}\big]_1(f)\bigg]_0\bigg( \bigg[\big[\Sigma^{\C P, \C Q}\big]_1(g)\bigg]_0(z, w)\bigg)\\
& = \bigg[\big[\Sigma^{\C P, \C Q}\big]_1(f)\bigg]_0\bigg(\underbrace{[\C P_1(g)]_0(z)}_y, [\C Q_1(g, 1_y)]_0(w) \bigg)\\
& = \bigg(\underbrace{\C P_1(f)]_0(y)}_x,  [\C Q_1(f, 1_x)]_0\big([\C Q_1(g, 1_y)]_0(w)\big)\bigg)\\
& = \big(x, [\C Q_1(g \circ f, 1_x)]_0(w)\big)\\
& = C,
\end{align*}
as $(g, 1_y) \colon (b, y) \to (c, z)$ and $(f, 1_x) \colon (a, x) \to (b, y)$ in $\Sigma(\C C, \C P)$,
and hence
$$(g, 1_y) \circ (f, 1_x) = (g \circ f, [\C P_1(f)]_1(1_y) \circ 1_x) = (g \circ f, 1_{[\C P_1(f)]_0(y)}) =
(g \circ f, 1_x).$$
If $(j{''}, \mu{''}) \colon (y{'}, t{'}) \to (y{''}, t{''})$ in $\Sigma(\C P_0(c), \C Q^c)$, then
$$\big[\Sigma^{\C P, \C Q}_1(g)\big]_1(j{''}, \mu{''}) = 
\bigg(\underbrace{\big[\C P_1(g)\big]_1(j{''})}_{j{'}}, 
\underbrace{\big[Q_1(g, 1_{x{'}})\big]_1(\mu{''})}_{\mu{'}}\bigg),$$
where $x{'} = [\C P_1(g)]_0(y{'})$. Similarly,
$$\big[\Sigma^{\C P, \C Q}_1(f)\big]_1(j{'}, \mu{'}) = \bigg(\underbrace{\big[\C P_1(f)\big]_1(j{'})}_{j}, 
\underbrace{\big[Q_1(g, 1_x)\big]_1(\mu{'})}_{\mu}\bigg),$$
where $x = [\C P_1(g)]_0(x{'})$. Similarly,
$$\big[\Sigma^{\C P, \C Q}_1(g \circ f)\big]_1(j{''}, \mu{''}) = \bigg(\underbrace{\big[\C P_1(g 
\circ f)\big]_1(j{''})}_{j}, \underbrace{\big[Q_1(g \circ f, 1_x)\big]_1(\mu{''})}_{\mu}\bigg),$$
as $(g, 1_{x{'}}) \colon (b, x{'}) \to (c, y{'})$, $(f, 1_x) \colon (a, x) \to (b, x{'})$ and 
$$(g, 1_{x{'}}) \circ (f, 1_x) = \big(g \circ f, [\C P_1(f)]_1(1_{x{'}}) \circ 1_x\big) = 
\big(g \circ f, 1_{[\C P_1(f)]_0(x{'})} \circ 1_x\big) = (g \circ f, 1_x \circ 1_x).$$
(iii) The objects of the category $\Sigma\big(\C C, \Sigma^{\C P, \C Q}\big)$ is the set 
$$\sum_{a \in \Ob_{\C C}}\Ob_{\Sigma^{\C P, \C Q}_0(a)} = 
\sum_{a \in \Ob_{\C C}}\Ob_{\Sigma(\C P_0(a), \C Q^a)} 
= \sum_{a \in \Ob_{\C C}}\sum_{x \in \Ob_{\C P_0(a)}}\Ob_{\C Q_0(a, x)}.$$
A morphism in $\Sigma\big(\C C, \Sigma^{\C P, \C Q}\big)$ is a pair $(f, (i, \lambda))
\colon (a, (x, u)) \to (b, (y, v))$,
where $f \colon a \to b$ in $\C C$ and $(i, \lambda) \colon (x, u) \to 
\big[\Sigma^{\C P, \C Q}_1(f)\big]_0(y, v)$ in 
$\Sigma(\C P_0(a), \C Q^a)$ with
$$\big[\Sigma^{\C P, \C Q}_1(f)\big]_0(y, v) = \bigg(\underbrace{[\C P_1(f)]_0(y)}_{x{'}}, 
[\C Q_1(f, 1_{x{'}})]_0(v)\bigg).$$
Hence $i \colon x \to x{'}$ in $\C P_0(a)$ and 
$\lambda \colon u \to [\C Q_1(1_a, i)]_0\big([\C Q_1(f, 1_{x{'}})]_0(v)\big)$ in $\C Q_0(a, x)$.
By the proof of (i) $(1_a, i) \colon (a, x) \to (a, x{'})$ in $\Sigma(\C C, \C P)$, and 
$(f, 1_{x{'}}) \colon (a, x{'}) \to (b, y)$ in $\Sigma(\C C, \C P)$. Hence 
$$(f, 1_{x{'}}) \circ (1_a, i) = \big(f \circ 1_a, [\C P_a(1_a)]_1(1_{x{'}}) \circ i\big) = 
\big(f, \big[1_{\C P_0(a)}\big]_1(1_{_x{'}}) \circ i\big) = (f, 1_{x{'}} \circ i) = (f, i).$$
Consequently, $\lambda \colon u \to [\C Q_1(f, i)]_0(v)$.
The objects of the category $\Sigma\big(\Sigma(\C C, \C P), \C Q\big)$ is the set
$$\sum_{(a, x) \in \sum_{a \in \Ob_{\C C}}\Ob_{\C P_0(a)}}\Ob_{\C Q_0(a, x)},$$
while a morphism in $\Sigma\big(\Sigma(\C C, \C P), \C Q\big)$ is a pair $((f, i), \lambda)
\colon ((a, x), u) \to
((b, y), v)$, where $(f, i) \colon (a, x) \to (b, y)$ is a morphism in $\Sigma(\C C, \C P)$ 
i.e., $f \colon a \to b$ in
$\C C$, $i \colon x \to [\C P_1(f)]_0(y)$ in $\C P_0(a)$, and $\lambda \colon u \to [\C Q_1(f, i)]_0(v)$
in $\C Q_0(a, x)$.

Let $F \colon \Sigma\big(\C C, \Sigma^{\C P, \C Q}\big) \to \Sigma\big(\Sigma(\C C, \C P), \C Q\big)$, where
$F_0(a, (x, u)) = ((a, x), u)$ and $F_1(f, (i, \lambda)) = ((f, i), \lambda)$. We show that $F$ is a 
functor.
The fact that $F_1$ preserves units 
is trivial.

Suppose next that $(f, (i, \lambda)) \colon (a, (x, u)) \to (b, (y, v))$,
where $f \colon a \to b$ in $\C C$, $i \colon x \to x{'}$ in $\C P_0(a)$, where $x{'} = [\C P_1(f)]_0(y)$,
and
$\lambda \colon u \to [\C Q_1(f, i)]_0(v)$ in $\C Q_0(a, x)$. Let also 
$(g, (j, \mu)) \colon (b, (y, v)) \to (c, (z, w))$,
where $g \colon b \to c$ in $\C C$, $j \colon y \to y{'}$ in $\C P_0(b)$, where $y{'} = [\C P_1(g)]_0(z)$,
and
$\mu \colon v \to [\C Q_1(g, j)]_0(w)$ in $\C Q_0(b, y)$. Their composition in $\Sigma\big(\C C, 
\Sigma^{\C P, \C Q}\big)$
is the morphism 
$$\bigg(g \circ f, \big[\Sigma^{\C P, \C Q}_1(f)\big]_1(j, \mu) \circ (i, \lambda)\bigg),$$
where the above composition is in $\Sigma(\C P_0(a), \C Q^a)$ and $\Sigma^{\C P, \C Q}_1(f)\big]_1(j, \mu)
= (j{'}, \mu{'})$,
where $j{'} = [\C P_a(f)]_1(j)$ and $\mu{'} = [\C Q_1(f, 1_{x{'}})]_1(\mu)$, since $[\C P_1(f)]_0(y)$ 
is denoted in 
the above context by $x{'}$. Hence,
$$(j{'}, \mu{'}) \circ (i, \lambda) = \big(j{'} \circ i, [\C Q_1^a(i)]_1(\mu{'}) \circ \lambda\big) = 
\big(j{'} \circ i, [\C Q_1(1_a, i)]_1(\mu{'}) \circ \lambda\big), \ \ \mbox{and let} \ $$
\begin{align*}
 \ \ \ \ \ \ \ \ \ \ \ \ \ \ \ \ \ K & = F\big((g, (j, \mu)) \circ (f, (i, \lambda))\big)\\
& = \bigg(\big(g \circ f, j{'} \circ i), [\C Q_1(1_a, i)]_1(\mu{'}) \circ \lambda\bigg)\\
& = \bigg(\big(g \circ f, j{'} \circ i), [\C Q_1(1_a, i)]_1\big([\C Q_1(f, 1_{x{'}})]_1(\mu)\big) 
\circ \lambda\bigg)\\
& = \bigg(\big(g \circ f, j{'} \circ i), [\C Q_1(f, i]_1(\mu)\big) \circ \lambda\bigg),
\end{align*}
since, as we have shown above, $(f, 1_{x{'}}) \circ (1_a, i) = (f, i)$ in $\Sigma(\C C, \C P)$.
In $\Sigma\big(\Sigma(\C C, \C P), \C Q)$ 
\begin{align*}
 \ \ \ \ \ \ \ \ \ \ \ \ \ \ \ \ \ L & = F(g, (j, \mu)) \circ F(f, (i, \lambda)) \\
& = ((g, j), \mu) \circ ((f, i), \lambda)\\
& = \bigg((g, j) \circ (f, i), [\C Q_1(f, i)]_1(\mu) \circ \lambda\bigg)\\
& = \bigg(\big(g \circ f, [\C P_1(f)]_1(j) \circ i\big), [\C Q_1(f, i)]_1(\mu) \circ \lambda\bigg)\\
& = K.
\end{align*}
For the functor $G \colon \Sigma\big(\Sigma(\C C, \C P), \C Q\big) \to \Sigma\big(\C C, 
\Sigma^{\C P, \C Q}\big)$,
defined by $G_0((a, x), u) = (a, (x, u))$ and $G_1((f, i), \lambda) = (f, (i, \lambda))$, we
proceed similarly. It is immediate to show that the pair $(F,G)$ is an isomorphism of categories.
\end{proof}

\section{Concluding remarks}
\label{sec: concluding}

Here we presented two non-trivial examples of translating equivalences in $\MLTT$ that involve
the $\Sigma$-type to 
isomorphisms of categories that involve the Grothendieck construction. There are also many simple
examples
of this phenomenon, like the isomorphism of the categories $\Sigma(\C C, P) \times \Sigma(\C D, Q)$
and $\Sigma(\C C \times \C D, P \times Q)$ 
(or the categories $\Sigma(\C C, \C P) \times \Sigma(\C D, \C Q)$ and $\Sigma(\C C \times \C D, \C P
\times \C Q)$), that 
are not developed here. We expect to include other interesting instances 
of this phenomenon in future work. E.g., type-theoretic equivalences involving the $\Sigma$-type and 
the fiber of a function,
like $\pr_1$ (see e.g., section 4.8 in~\cite{HoTT13}), are expected to be translated into the language 
of the 
Grothendieck construction, since the fiber of a functor can be constructed as a pullback 
(see e.g.,~\cite{Ja99}, pp.~26-27). 
As the Grothendieck construction can be generalised to higher category theory (see e.g.,~\cite{MG15}), the 
extension of the aforementioned phenomenon to higher category theory is also expected. \\[1mm]

\noindent
\textbf{Acknowledgement}\\[1mm]
This research 
was supported by LMUexcellent, funded by the Federal
Ministry of Education and Research (BMBF) and the Free State of Bavaria under the
Excellence Strategy of the Federal Government and the L\"ander.

\end{document}